\documentclass[12pt]{article}
\usepackage{hyperref}
\usepackage{selectp}
\usepackage{graphicx}
\usepackage{multirow}
\usepackage{setspace}
\usepackage{amsmath,amssymb,amsthm}
\usepackage{url}
\usepackage[round]{natbib}
\usepackage{comment}
\usepackage{tikz}

\def\i{{\rm i}}

\newtheorem{theorem}{Theorem}[section]
\newtheorem{lemma}[theorem]{Lemma}

\def\myspace{\vskip 10pt}

\newcommand\marginal[1]{\marginpar{\raggedright\parindent=0pt\tiny #1}}
\newcommand\SJ{\marginal{SJ}}
\newcommand\JF{\marginal{JF}}
\newcommand\MDW{\marginal{MDW}}
\newcommand{\sfrac}[2]{{\textstyle\frac{#1}{#2}}}

\title{Partitions with Distinct Multiplicities of Parts:
On An ``Unsolved Problem'' Posed By Herbert Wilf}

\author{James Allen Fill%
\thanks{J.~A.~Fill's research was supported by the Acheson~J.~Duncan Fund for the Advancement of Research in
Statistics.}\\
\small Department of Applied Mathematics and Statistics\\[-0.8ex]
\small The Johns Hopkins University\\[-0.8ex]
\small 3400 N.~Charles Street\\[-0.8ex]
\small Baltimore, MD, USA 21218-2682\\[-0.8ex]
\small \texttt{jimfill@jhu.edu}\\[-0.8ex]
\small \texttt{http://www.ams.jhu.edu/}$\sim$\texttt{fill/}\\
\and Svante Janson\\
\small Department of Mathematics\\[-0.8ex]
\small Uppsala University\\[-0.8ex]
\small P.~O.~Box 480\\[-0.8ex]
\small SE-751 06 Uppsala, Sweden\\[-0.8ex]
\small \texttt{svante.janson@math.uu.se}\\[-0.8ex]
\small \texttt{http://www.math.uu.se/}$\sim$\texttt{svante/}\\
\and Mark Daniel Ward\thanks{M.~D.~Ward's research was supported by NSF
  Science \& Technology Center grant CCF-0939370.}\\
\small Department of Statistics\\[-0.8ex]
\small Purdue University\\[-0.8ex]
\small 250 N.~University Street\\[-0.8ex]
\small West Lafayette, IN, USA 47907-2066\\[-0.8ex]
\small \texttt{mdw@purdue.edu}\\[-0.8ex]
\small \texttt{http://www.stat.purdue.edu/}$\sim$\texttt{mdw/}
}

\begin{document}

\maketitle


\begin{abstract}
Wilf's Sixth Unsolved Problem asks for any interesting properties of the
set of partitions of integers for which the (nonzero)
multiplicities of the parts are all different.
We refer to these as \emph{Wilf partitions}.
Using $f(n)$ to denote the number of Wilf partitions,
we establish lead-order asymptotics for $\ln{f(n)}$.

\end{abstract}


\centerline{\emph{Dedicated to the memory of Herbert S.~Wilf (1931--2012).}}

\section{The Problem}
Herbert S.~Wilf was an expert in many areas of combinatorics.
Besides writing numerous papers and books, he was a friend
and mentor to many colleagues.  Herb often asked intriguing questions
that opened up whole new areas of investigation.  In the later years
of his life, he posted a set of eight Unsolved Problems on his
webpage~\cite{WilfPDF}.
At the time of Herb's death in January 2012, 
only one of these problems was solved (the third problem:\ see~\cite{Ward2010}).
In this paper, we discuss Wilf's sixth ``Unsolved Problem.''

\begin{center}\fbox{
      \begin{minipage}{16cm}
\textbf{Distinct multiplicities} \cite{WilfPDF}

Let $T(n)$ be the set of partitions of $n$ for which the (nonzero)
multiplicities of its parts are all different, and write $f(n) =
|T(n)|$. See Sloane's sequence {\tt A098859} for a table of
values. Find any interesting theorems about $f(n)$. The mapping that
sends a partition of~$n$ to another partition of~$n$ in which the
roles of parts and multiplicities are interchanged is a well defined
involution on $T(n)$, which is how I arrived at the study of this
problem.
      \end{minipage}
    }
\end{center}

\section{Definitions}
We refer to partitions in which the (nonzero) multiplicities of the
parts are all different as \emph{Wilf partitions}.

Define
$$\mathcal{M}_{r} := \{(m_{1},m_{2},\ldots,m_{r}):\hbox{$m_{k}$ are
  distinct positive integers}\},$$
and
$$\mathcal{P}_{r} := \{(p_{1},p_{2},\ldots,p_{r}):\hbox{$p_{k}$ are
  distinct positive integers with $p_{1}<\cdots<p_{r}$}\}.$$
Then the set of Wilf partitions of~$n$ is
$$T(n) := \bigcup_{r\geq 1}
\{ ({\bf m}, {\bf p}):{\bf m} = (m_{1},\ldots,m_{r});\ 
{\bf p} = (p_{1},\ldots,p_{r});\ 
m_{1}p_{1}+\cdots+m_{r}p_{r}=n\}.$$

\myspace
We write $p(n, r)$ for the number of partitions of $n$ into 
$r$ parts, and $d(\cdot)$ to denote the divisor function, i.e., $d(n)$
is the number of divisors of $n$.  

\myspace
Wilf defines $f(n) = |T(n)|$ and then asks to find anything
interesting about $T(n)$.  

\section{Main Result}
\begin{theorem}
Let $f(n)$ denote the number of partitions of $n$ with distinct
(nonzero) multiplicities.  Then
$$
\ln{f(n)} \sim \frac{6^{1/3}}{3} n^{1/3}\ln{n} 
\sim
(6n)^{1/3}\ln[(6n)^{1/3}]
\qquad \text{\rm as\ }n\to\infty.
$$
\end{theorem}
The theorem will be established by matching upper and lower bounds for
$\ln{f(n)}$
in Lemmas \ref{upperbound} and \ref{lowerbound} below.

In Figure~\ref{valuesplot},
we plot the values of
$$
\frac{\ln f(n)}{(6n)^{1/3}\ln[(6n)^{1/3}]}
$$ 
for $31 \leq n \leq 508$.
\begin{center}
\begin{figure}
\begin{tikzpicture}[xscale=.03,yscale=20]
\draw[->] (-.5,.68) -- (511,.68) node[right] {$x$} coordinate(x axis);
\draw[->] (0,.68) -- (0,.92) node[above] {$y$} coordinate(y axis);
\foreach \y in {.70,.75,.80,.85,.90}
  \draw (3, \y) -- (-3, \y) node[anchor=east] {$\y$};
\foreach \x in {50,100,150,200,250,300,350,400,450,500}
  \draw (\x, .685) -- (\x, .675) node[anchor=north] {$\x$};
\draw (-.5,.67) node {$0$};
\draw[black, ultra thick, fill=black] (31,0.706110637128887) ellipse(.06 and .06/500);
\draw[black, ultra thick, fill=black] (32,0.705240876035161) ellipse(.06 and .06/500);
\draw[black, ultra thick, fill=black] (33,0.710169863906683) ellipse(.06 and .06/500);
\draw[black, ultra thick, fill=black] (34,0.711433042128664) ellipse(.06 and .06/500);
\draw[black, ultra thick, fill=black] (35,0.717596036249754) ellipse(.06 and .06/500);
\draw[black, ultra thick, fill=black] (36,0.715338391559701) ellipse(.06 and .06/500);
\draw[black, ultra thick, fill=black] (37,0.721981012739142) ellipse(.06 and .06/500);
\draw[black, ultra thick, fill=black] (38,0.721771381665251) ellipse(.06 and .06/500);
\draw[black, ultra thick, fill=black] (39,0.725898316044059) ellipse(.06 and .06/500);
\draw[black, ultra thick, fill=black] (40,0.72626853249332) ellipse(.06 and .06/500);
\draw[black, ultra thick, fill=black] (41,0.731192162699007) ellipse(.06 and .06/500);
\draw[black, ultra thick, fill=black] (42,0.730034588902367) ellipse(.06 and .06/500);
\draw[black, ultra thick, fill=black] (43,0.735281478457999) ellipse(.06 and .06/500);
\draw[black, ultra thick, fill=black] (44,0.734795628441338) ellipse(.06 and .06/500);
\draw[black, ultra thick, fill=black] (45,0.738733091353069) ellipse(.06 and .06/500);
\draw[black, ultra thick, fill=black] (46,0.73867108802635) ellipse(.06 and .06/500);
\draw[black, ultra thick, fill=black] (47,0.742900185560511) ellipse(.06 and .06/500);
\draw[black, ultra thick, fill=black] (48,0.742257467561187) ellipse(.06 and .06/500);
\draw[black, ultra thick, fill=black] (49,0.746284392717777) ellipse(.06 and .06/500);
\draw[black, ultra thick, fill=black] (50,0.746100684698997) ellipse(.06 and .06/500);
\draw[black, ultra thick, fill=black] (51,0.749441612317767) ellipse(.06 and .06/500);
\draw[black, ultra thick, fill=black] (52,0.749441773965944) ellipse(.06 and .06/500);
\draw[black, ultra thick, fill=black] (53,0.752831279782746) ellipse(.06 and .06/500);
\draw[black, ultra thick, fill=black] (54,0.752680630156736) ellipse(.06 and .06/500);
\draw[black, ultra thick, fill=black] (55,0.755906791356821) ellipse(.06 and .06/500);
\draw[black, ultra thick, fill=black] (56,0.755791700053252) ellipse(.06 and .06/500);
\draw[black, ultra thick, fill=black] (57,0.758582125048533) ellipse(.06 and .06/500);
\draw[black, ultra thick, fill=black] (58,0.758631376040015) ellipse(.06 and .06/500);
\draw[black, ultra thick, fill=black] (59,0.761468206982315) ellipse(.06 and .06/500);
\draw[black, ultra thick, fill=black] (60,0.761505832916969) ellipse(.06 and .06/500);
\draw[black, ultra thick, fill=black] (61,0.763960103967583) ellipse(.06 and .06/500);
\draw[black, ultra thick, fill=black] (62,0.764060843057156) ellipse(.06 and .06/500);
\draw[black, ultra thick, fill=black] (63,0.76658317244959) ellipse(.06 and .06/500);
\draw[black, ultra thick, fill=black] (64,0.766595704715914) ellipse(.06 and .06/500);
\draw[black, ultra thick, fill=black] (65,0.768904986653663) ellipse(.06 and .06/500);
\draw[black, ultra thick, fill=black] (66,0.769062358919275) ellipse(.06 and .06/500);
\draw[black, ultra thick, fill=black] (67,0.771131053130279) ellipse(.06 and .06/500);
\draw[black, ultra thick, fill=black] (68,0.771340599247202) ellipse(.06 and .06/500);
\draw[black, ultra thick, fill=black] (69,0.773422932783723) ellipse(.06 and .06/500);
\draw[black, ultra thick, fill=black] (70,0.773547042448312) ellipse(.06 and .06/500);
\draw[black, ultra thick, fill=black] (71,0.775485081246977) ellipse(.06 and .06/500);
\draw[black, ultra thick, fill=black] (72,0.775818344920678) ellipse(.06 and .06/500);
\draw[black, ultra thick, fill=black] (73,0.777498782177007) ellipse(.06 and .06/500);
\draw[black, ultra thick, fill=black] (74,0.777744760580408) ellipse(.06 and .06/500);
\draw[black, ultra thick, fill=black] (75,0.77962694855662) ellipse(.06 and .06/500);
\draw[black, ultra thick, fill=black] (76,0.779772984750158) ellipse(.06 and .06/500);
\draw[black, ultra thick, fill=black] (77,0.781403499179355) ellipse(.06 and .06/500);
\draw[black, ultra thick, fill=black] (78,0.781792868529785) ellipse(.06 and .06/500);
\draw[black, ultra thick, fill=black] (79,0.783257777666823) ellipse(.06 and .06/500);
\draw[black, ultra thick, fill=black] (80,0.78361168023681) ellipse(.06 and .06/500);
\draw[black, ultra thick, fill=black] (81,0.785159181854192) ellipse(.06 and .06/500);
\draw[black, ultra thick, fill=black] (82,0.785364019319305) ellipse(.06 and .06/500);
\draw[black, ultra thick, fill=black] (83,0.786809720247371) ellipse(.06 and .06/500);
\draw[black, ultra thick, fill=black] (84,0.787269531199771) ellipse(.06 and .06/500);
\draw[black, ultra thick, fill=black] (85,0.788493942155313) ellipse(.06 and .06/500);
\draw[black, ultra thick, fill=black] (86,0.788828495319814) ellipse(.06 and .06/500);
\draw[black, ultra thick, fill=black] (87,0.790227518520125) ellipse(.06 and .06/500);
\draw[black, ultra thick, fill=black] (88,0.79049476160637) ellipse(.06 and .06/500);
\draw[black, ultra thick, fill=black] (89,0.791705186389935) ellipse(.06 and .06/500);
\draw[black, ultra thick, fill=black] (90,0.792189319399616) ellipse(.06 and .06/500);
\draw[black, ultra thick, fill=black] (91,0.793245610931814) ellipse(.06 and .06/500);
\draw[black, ultra thick, fill=black] (92,0.793636178906206) ellipse(.06 and .06/500);
\draw[black, ultra thick, fill=black] (93,0.79483005226509) ellipse(.06 and .06/500);
\draw[black, ultra thick, fill=black] (94,0.795108534749934) ellipse(.06 and .06/500);
\draw[black, ultra thick, fill=black] (95,0.796200699192181) ellipse(.06 and .06/500);
\draw[black, ultra thick, fill=black] (96,0.796680937209092) ellipse(.06 and .06/500);
\draw[black, ultra thick, fill=black] (97,0.797586177524331) ellipse(.06 and .06/500);
\draw[black, ultra thick, fill=black] (98,0.797984214964852) ellipse(.06 and .06/500);
\draw[black, ultra thick, fill=black] (99,0.799049424668919) ellipse(.06 and .06/500);
\draw[black, ultra thick, fill=black] (100,0.799379630033019) ellipse(.06 and .06/500);
\draw[black, ultra thick, fill=black] (101,0.800289319010827) ellipse(.06 and .06/500);
\draw[black, ultra thick, fill=black] (102,0.800772665334315) ellipse(.06 and .06/500);
\draw[black, ultra thick, fill=black] (103,0.801584572052385) ellipse(.06 and .06/500);
\draw[black, ultra thick, fill=black] (104,0.802004198694052) ellipse(.06 and .06/500);
\draw[black, ultra thick, fill=black] (105,0.80293278072025) ellipse(.06 and .06/500);
\draw[black, ultra thick, fill=black] (106,0.803253336826086) ellipse(.06 and .06/500);
\draw[black, ultra thick, fill=black] (107,0.804077308420786) ellipse(.06 and .06/500);
\draw[black, ultra thick, fill=black] (108,0.804573881036168) ellipse(.06 and .06/500);
\draw[black, ultra thick, fill=black] (109,0.805273996843882) ellipse(.06 and .06/500);
\draw[black, ultra thick, fill=black] (110,0.805696214285108) ellipse(.06 and .06/500);
\draw[black, ultra thick, fill=black] (111,0.806518254891049) ellipse(.06 and .06/500);
\draw[black, ultra thick, fill=black] (112,0.806874959639077) ellipse(.06 and .06/500);
\draw[black, ultra thick, fill=black] (113,0.807592020333306) ellipse(.06 and .06/500);
\draw[black, ultra thick, fill=black] (114,0.808074452465637) ellipse(.06 and .06/500);
\draw[black, ultra thick, fill=black] (115,0.808716464790315) ellipse(.06 and .06/500);
\draw[black, ultra thick, fill=black] (116,0.809139488737431) ellipse(.06 and .06/500);
\draw[black, ultra thick, fill=black] (117,0.809863477996509) ellipse(.06 and .06/500);
\draw[black, ultra thick, fill=black] (118,0.810223566491169) ellipse(.06 and .06/500);
\draw[black, ultra thick, fill=black] (119,0.810875382950992) ellipse(.06 and .06/500);
\draw[black, ultra thick, fill=black] (120,0.811362005876159) ellipse(.06 and .06/500);
\draw[black, ultra thick, fill=black] (121,0.811922121960845) ellipse(.06 and .06/500);
\draw[black, ultra thick, fill=black] (122,0.812340594313567) ellipse(.06 and .06/500);
\draw[black, ultra thick, fill=black] (123,0.812997389056019) ellipse(.06 and .06/500);
\draw[black, ultra thick, fill=black] (124,0.813370739266003) ellipse(.06 and .06/500);
\draw[black, ultra thick, fill=black] (125,0.813950922072906) ellipse(.06 and .06/500);
\draw[black, ultra thick, fill=black] (126,0.814414052704507) ellipse(.06 and .06/500);
\draw[black, ultra thick, fill=black] (127,0.814930468584504) ellipse(.06 and .06/500);
\draw[black, ultra thick, fill=black] (128,0.815352133658764) ellipse(.06 and .06/500);
\draw[black, ultra thick, fill=black] (129,0.81593525506633) ellipse(.06 and .06/500);
\draw[black, ultra thick, fill=black] (130,0.816307719452679) ellipse(.06 and .06/500);
\draw[black, ultra thick, fill=black] (131,0.816832503845985) ellipse(.06 and .06/500);
\draw[black, ultra thick, fill=black] (132,0.817291240156789) ellipse(.06 and .06/500);
\draw[black, ultra thick, fill=black] (133,0.817755281731513) ellipse(.06 and .06/500);
\draw[black, ultra thick, fill=black] (134,0.818164520706203) ellipse(.06 and .06/500);
\draw[black, ultra thick, fill=black] (135,0.818698600200935) ellipse(.06 and .06/500);
\draw[black, ultra thick, fill=black] (136,0.819070967097392) ellipse(.06 and .06/500);
\draw[black, ultra thick, fill=black] (137,0.81954402391933) ellipse(.06 and .06/500);
\draw[black, ultra thick, fill=black] (138,0.81998333895098) ellipse(.06 and .06/500);
\draw[black, ultra thick, fill=black] (139,0.820412384612539) ellipse(.06 and .06/500);
\draw[black, ultra thick, fill=black] (140,0.820818919670332) ellipse(.06 and .06/500);
\draw[black, ultra thick, fill=black] (141,0.821296181237362) ellipse(.06 and .06/500);
\draw[black, ultra thick, fill=black] (142,0.821662836442879) ellipse(.06 and .06/500);
\draw[black, ultra thick, fill=black] (143,0.822097514760091) ellipse(.06 and .06/500);
\draw[black, ultra thick, fill=black] (144,0.822526708405256) ellipse(.06 and .06/500);
\draw[black, ultra thick, fill=black] (145,0.822919028651353) ellipse(.06 and .06/500);
\draw[black, ultra thick, fill=black] (146,0.823309394197762) ellipse(.06 and .06/500);
\draw[black, ultra thick, fill=black] (147,0.823747991604472) ellipse(.06 and .06/500);
\draw[black, ultra thick, fill=black] (148,0.824113022352538) ellipse(.06 and .06/500);
\draw[black, ultra thick, fill=black] (149,0.824509104149123) ellipse(.06 and .06/500);
\draw[black, ultra thick, fill=black] (150,0.824920182345627) ellipse(.06 and .06/500);
\draw[black, ultra thick, fill=black] (151,0.825284334397209) ellipse(.06 and .06/500);
\draw[black, ultra thick, fill=black] (152,0.825668085760619) ellipse(.06 and .06/500);
\draw[black, ultra thick, fill=black] (153,0.826067249359654) ellipse(.06 and .06/500);
\draw[black, ultra thick, fill=black] (154,0.826422974538604) ellipse(.06 and .06/500);
\draw[black, ultra thick, fill=black] (155,0.826791090111328) ellipse(.06 and .06/500);
\draw[black, ultra thick, fill=black] (156,0.827189538565963) ellipse(.06 and .06/500);
\draw[black, ultra thick, fill=black] (157,0.827525823332532) ellipse(.06 and .06/500);
\draw[black, ultra thick, fill=black] (158,0.827896354498738) ellipse(.06 and .06/500);
\draw[black, ultra thick, fill=black] (159,0.828265478106706) ellipse(.06 and .06/500);
\draw[black, ultra thick, fill=black] (160,0.828617321495158) ellipse(.06 and .06/500);
\draw[black, ultra thick, fill=black] (161,0.828956006747748) ellipse(.06 and .06/500);
\draw[black, ultra thick, fill=black] (162,0.829338040110526) ellipse(.06 and .06/500);
\draw[black, ultra thick, fill=black] (163,0.829653889970261) ellipse(.06 and .06/500);
\draw[black, ultra thick, fill=black] (164,0.830015829860727) ellipse(.06 and .06/500);
\draw[black, ultra thick, fill=black] (165,0.830356336882051) ellipse(.06 and .06/500);
\draw[black, ultra thick, fill=black] (166,0.83069779763364) ellipse(.06 and .06/500);
\draw[black, ultra thick, fill=black] (167,0.831014592898181) ellipse(.06 and .06/500);
\draw[black, ultra thick, fill=black] (168,0.831384651786174) ellipse(.06 and .06/500);
\draw[black, ultra thick, fill=black] (169,0.831680280951399) ellipse(.06 and .06/500);
\draw[black, ultra thick, fill=black] (170,0.832029730611298) ellipse(.06 and .06/500);
\draw[black, ultra thick, fill=black] (171,0.832346716951683) ellipse(.06 and .06/500);
\draw[black, ultra thick, fill=black] (172,0.832682070347819) ellipse(.06 and .06/500);
\draw[black, ultra thick, fill=black] (173,0.832977238690333) ellipse(.06 and .06/500);
\draw[black, ultra thick, fill=black] (174,0.83333227215135) ellipse(.06 and .06/500);
\draw[black, ultra thick, fill=black] (175,0.833611860447188) ellipse(.06 and .06/500);
\draw[black, ultra thick, fill=black] (176,0.833951805701951) ellipse(.06 and .06/500);
\draw[black, ultra thick, fill=black] (177,0.834246928029212) ellipse(.06 and .06/500);
\draw[black, ultra thick, fill=black] (178,0.834572035702212) ellipse(.06 and .06/500);
\draw[black, ultra thick, fill=black] (179,0.834850172982628) ellipse(.06 and .06/500);
\draw[black, ultra thick, fill=black] (180,0.835193465500406) ellipse(.06 and .06/500);
\draw[black, ultra thick, fill=black] (181,0.835457206057912) ellipse(.06 and .06/500);
\draw[black, ultra thick, fill=black] (182,0.835784866180839) ellipse(.06 and .06/500);
\draw[black, ultra thick, fill=black] (183,0.836061835902081) ellipse(.06 and .06/500);
\draw[black, ultra thick, fill=black] (184,0.836379376998579) ellipse(.06 and .06/500);
\draw[black, ultra thick, fill=black] (185,0.836641083156803) ellipse(.06 and .06/500);
\draw[black, ultra thick, fill=black] (186,0.836970561550972) ellipse(.06 and .06/500);
\draw[black, ultra thick, fill=black] (187,0.837220983310418) ellipse(.06 and .06/500);
\draw[black, ultra thick, fill=black] (188,0.837538896993402) ellipse(.06 and .06/500);
\draw[black, ultra thick, fill=black] (189,0.837798916062862) ellipse(.06 and .06/500);
\draw[black, ultra thick, fill=black] (190,0.838106379807468) ellipse(.06 and .06/500);
\draw[black, ultra thick, fill=black] (191,0.838354320373869) ellipse(.06 and .06/500);
\draw[black, ultra thick, fill=black] (192,0.838672160941738) ellipse(.06 and .06/500);
\draw[black, ultra thick, fill=black] (193,0.838910098510306) ellipse(.06 and .06/500);
\draw[black, ultra thick, fill=black] (194,0.839216462770823) ellipse(.06 and .06/500);
\draw[black, ultra thick, fill=black] (195,0.839462018227143) ellipse(.06 and .06/500);
\draw[black, ultra thick, fill=black] (196,0.8397611334322) ellipse(.06 and .06/500);
\draw[black, ultra thick, fill=black] (197,0.839995978809832) ellipse(.06 and .06/500);
\draw[black, ultra thick, fill=black] (198,0.840301136069842) ellipse(.06 and .06/500);
\draw[black, ultra thick, fill=black] (199,0.840528181664755) ellipse(.06 and .06/500);
\draw[black, ultra thick, fill=black] (200,0.840824810194428) ellipse(.06 and .06/500);
\draw[black, ultra thick, fill=black] (201,0.841057187065222) ellipse(.06 and .06/500);
\draw[black, ultra thick, fill=black] (202,0.84134627656073) ellipse(.06 and .06/500);
\draw[black, ultra thick, fill=black] (203,0.841569704059137) ellipse(.06 and .06/500);
\draw[black, ultra thick, fill=black] (204,0.841864099963387) ellipse(.06 and .06/500);
\draw[black, ultra thick, fill=black] (205,0.842081049811931) ellipse(.06 and .06/500);
\draw[black, ultra thick, fill=black] (206,0.84236698083317) ellipse(.06 and .06/500);
\draw[black, ultra thick, fill=black] (207,0.842587581439577) ellipse(.06 and .06/500);
\draw[black, ultra thick, fill=black] (208,0.842868015253047) ellipse(.06 and .06/500);
\draw[black, ultra thick, fill=black] (209,0.843081197406911) ellipse(.06 and .06/500);
\draw[black, ultra thick, fill=black] (210,0.84336397775686) ellipse(.06 and .06/500);
\draw[black, ultra thick, fill=black] (211,0.843572047715916) ellipse(.06 and .06/500);
\draw[black, ultra thick, fill=black] (212,0.843848513508718) ellipse(.06 and .06/500);
\draw[black, ultra thick, fill=black] (213,0.844058535493104) ellipse(.06 and .06/500);
\draw[black, ultra thick, fill=black] (214,0.844329624109845) ellipse(.06 and .06/500);
\draw[black, ultra thick, fill=black] (215,0.844533479426807) ellipse(.06 and .06/500);
\draw[black, ultra thick, fill=black] (216,0.844806367074699) ellipse(.06 and .06/500);
\draw[black, ultra thick, fill=black] (217,0.845006075454512) ellipse(.06 and .06/500);
\draw[black, ultra thick, fill=black] (218,0.845272615066781) ellipse(.06 and .06/500);
\draw[black, ultra thick, fill=black] (219,0.845473264699724) ellipse(.06 and .06/500);
\draw[black, ultra thick, fill=black] (220,0.845735881079094) ellipse(.06 and .06/500);
\draw[black, ultra thick, fill=black] (221,0.845931619238971) ellipse(.06 and .06/500);
\draw[black, ultra thick, fill=black] (222,0.846194091918313) ellipse(.06 and .06/500);
\draw[black, ultra thick, fill=black] (223,0.846386242842779) ellipse(.06 and .06/500);
\draw[black, ultra thick, fill=black] (224,0.846643942147361) ellipse(.06 and .06/500);
\draw[black, ultra thick, fill=black] (225,0.846836197349097) ellipse(.06 and .06/500);
\draw[black, ultra thick, fill=black] (226,0.847090174149138) ellipse(.06 and .06/500);
\draw[black, ultra thick, fill=black] (227,0.847278183168817) ellipse(.06 and .06/500);
\draw[black, ultra thick, fill=black] (228,0.847531437632031) ellipse(.06 and .06/500);
\draw[black, ultra thick, fill=black] (229,0.847716855045482) ellipse(.06 and .06/500);
\draw[black, ultra thick, fill=black] (230,0.847965468093396) ellipse(.06 and .06/500);
\draw[black, ultra thick, fill=black] (231,0.848150139558112) ellipse(.06 and .06/500);
\draw[black, ultra thick, fill=black] (232,0.848396049026807) ellipse(.06 and .06/500);
\draw[black, ultra thick, fill=black] (233,0.848577287229476) ellipse(.06 and .06/500);
\draw[black, ultra thick, fill=black] (234,0.848821298787101) ellipse(.06 and .06/500);
\draw[black, ultra thick, fill=black] (235,0.84900029811414) ellipse(.06 and .06/500);
\draw[black, ultra thick, fill=black] (236,0.849240828726604) ellipse(.06 and .06/500);
\draw[black, ultra thick, fill=black] (237,0.849418626139366) ellipse(.06 and .06/500);
\draw[black, ultra thick, fill=black] (238,0.849656328038203) ellipse(.06 and .06/500);
\draw[black, ultra thick, fill=black] (239,0.849831226533833) ellipse(.06 and .06/500);
\draw[black, ultra thick, fill=black] (240,0.8500667689095) ellipse(.06 and .06/500);
\draw[black, ultra thick, fill=black] (241,0.850240197354317) ellipse(.06 and .06/500);
\draw[black, ultra thick, fill=black] (242,0.850472405713737) ellipse(.06 and .06/500);
\draw[black, ultra thick, fill=black] (243,0.850643800598496) ellipse(.06 and .06/500);
\draw[black, ultra thick, fill=black] (244,0.850873964156905) ellipse(.06 and .06/500);
\draw[black, ultra thick, fill=black] (245,0.851043139675054) ellipse(.06 and .06/500);
\draw[black, ultra thick, fill=black] (246,0.851270472979794) ellipse(.06 and .06/500);
\draw[black, ultra thick, fill=black] (247,0.851438351788748) ellipse(.06 and .06/500);
\draw[black, ultra thick, fill=black] (248,0.851662903551531) ellipse(.06 and .06/500);
\draw[black, ultra thick, fill=black] (249,0.851828666059565) ellipse(.06 and .06/500);
\draw[black, ultra thick, fill=black] (250,0.852051202771949) ellipse(.06 and .06/500);
\draw[black, ultra thick, fill=black] (251,0.852215056748923) ellipse(.06 and .06/500);
\draw[black, ultra thick, fill=black] (252,0.852434613398563) ellipse(.06 and .06/500);
\draw[black, ultra thick, fill=black] (253,0.85259756555496) ellipse(.06 and .06/500);
\draw[black, ultra thick, fill=black] (254,0.852814576529689) ellipse(.06 and .06/500);
\draw[black, ultra thick, fill=black] (255,0.85297487608836) ellipse(.06 and .06/500);
\draw[black, ultra thick, fill=black] (256,0.853190416366892) ellipse(.06 and .06/500);
\draw[black, ultra thick, fill=black] (257,0.853349367368359) ellipse(.06 and .06/500);
\draw[black, ultra thick, fill=black] (258,0.853561358890873) ellipse(.06 and .06/500);
\draw[black, ultra thick, fill=black] (259,0.853719504937801) ellipse(.06 and .06/500);
\draw[black, ultra thick, fill=black] (260,0.853929406542141) ellipse(.06 and .06/500);
\draw[black, ultra thick, fill=black] (261,0.854084986044061) ellipse(.06 and .06/500);
\draw[black, ultra thick, fill=black] (262,0.854293449671453) ellipse(.06 and .06/500);
\draw[black, ultra thick, fill=black] (263,0.854447673877359) ellipse(.06 and .06/500);
\draw[black, ultra thick, fill=black] (264,0.854652609490384) ellipse(.06 and .06/500);
\draw[black, ultra thick, fill=black] (265,0.854806403587871) ellipse(.06 and .06/500);
\draw[black, ultra thick, fill=black] (266,0.855009436361619) ellipse(.06 and .06/500);
\draw[black, ultra thick, fill=black] (267,0.855160312489851) ellipse(.06 and .06/500);
\draw[black, ultra thick, fill=black] (268,0.855362138345285) ellipse(.06 and .06/500);
\draw[black, ultra thick, fill=black] (269,0.855512141220382) ellipse(.06 and .06/500);
\draw[black, ultra thick, fill=black] (270,0.855710175378856) ellipse(.06 and .06/500);
\draw[black, ultra thick, fill=black] (271,0.85585977859272) ellipse(.06 and .06/500);
\draw[black, ultra thick, fill=black] (272,0.856056237374688) ellipse(.06 and .06/500);
\draw[black, ultra thick, fill=black] (273,0.856202936104425) ellipse(.06 and .06/500);
\draw[black, ultra thick, fill=black] (274,0.856398276046547) ellipse(.06 and .06/500);
\draw[black, ultra thick, fill=black] (275,0.856544104691467) ellipse(.06 and .06/500);
\draw[black, ultra thick, fill=black] (276,0.856735740295006) ellipse(.06 and .06/500);
\draw[black, ultra thick, fill=black] (277,0.856881470306025) ellipse(.06 and .06/500);
\draw[black, ultra thick, fill=black] (278,0.85707156767241) ellipse(.06 and .06/500);
\draw[black, ultra thick, fill=black] (279,0.857214212872762) ellipse(.06 and .06/500);
\draw[black, ultra thick, fill=black] (280,0.857403356926117) ellipse(.06 and .06/500);
\draw[black, ultra thick, fill=black] (281,0.85754551852711) ellipse(.06 and .06/500);
\draw[black, ultra thick, fill=black] (282,0.8577308829634) ellipse(.06 and .06/500);
\draw[black, ultra thick, fill=black] (283,0.857872813644724) ellipse(.06 and .06/500);
\draw[black, ultra thick, fill=black] (284,0.858056897384845) ellipse(.06 and .06/500);
\draw[black, ultra thick, fill=black] (285,0.858195877380214) ellipse(.06 and .06/500);
\draw[black, ultra thick, fill=black] (286,0.858379104181757) ellipse(.06 and .06/500);
\draw[black, ultra thick, fill=black] (287,0.858517571671865) ellipse(.06 and .06/500);
\draw[black, ultra thick, fill=black] (288,0.858697034894728) ellipse(.06 and .06/500);
\draw[black, ultra thick, fill=black] (289,0.858835532799895) ellipse(.06 and .06/500);
\draw[black, ultra thick, fill=black] (290,0.859013787287443) ellipse(.06 and .06/500);
\draw[black, ultra thick, fill=black] (291,0.859149248485706) ellipse(.06 and .06/500);
\draw[black, ultra thick, fill=black] (292,0.859326775671882) ellipse(.06 and .06/500);
\draw[black, ultra thick, fill=black] (293,0.85946193483804) ellipse(.06 and .06/500);
\draw[black, ultra thick, fill=black] (294,0.859635727482767) ellipse(.06 and .06/500);
\draw[black, ultra thick, fill=black] (295,0.859770817187534) ellipse(.06 and .06/500);
\draw[black, ultra thick, fill=black] (296,0.859943595855656) ellipse(.06 and .06/500);
\draw[black, ultra thick, fill=black] (297,0.860075829943812) ellipse(.06 and .06/500);
\draw[black, ultra thick, fill=black] (298,0.860247875138705) ellipse(.06 and .06/500);
\draw[black, ultra thick, fill=black] (299,0.860379782917348) ellipse(.06 and .06/500);
\draw[black, ultra thick, fill=black] (300,0.860548187337113) ellipse(.06 and .06/500);
\draw[black, ultra thick, fill=black] (301,0.860680217563965) ellipse(.06 and .06/500);
\draw[black, ultra thick, fill=black] (302,0.860847686630586) ellipse(.06 and .06/500);
\draw[black, ultra thick, fill=black] (303,0.860976766766274) ellipse(.06 and .06/500);
\draw[black, ultra thick, fill=black] (304,0.861143596750986) ellipse(.06 and .06/500);
\draw[black, ultra thick, fill=black] (305,0.861272527652483) ellipse(.06 and .06/500);
\draw[black, ultra thick, fill=black] (306,0.8614358156577) ellipse(.06 and .06/500);
\draw[black, ultra thick, fill=black] (307,0.861564786150533) ellipse(.06 and .06/500);
\draw[black, ultra thick, fill=black] (308,0.861727207593518) ellipse(.06 and .06/500);
\draw[black, ultra thick, fill=black] (309,0.861853424783699) ellipse(.06 and .06/500);
\draw[black, ultra thick, fill=black] (310,0.862015229539335) ellipse(.06 and .06/500);
\draw[black, ultra thick, fill=black] (311,0.862141267877435) ellipse(.06 and .06/500);
\draw[black, ultra thick, fill=black] (312,0.862299649329664) ellipse(.06 and .06/500);
\draw[black, ultra thick, fill=black] (313,0.862425820647065) ellipse(.06 and .06/500);
\draw[black, ultra thick, fill=black] (314,0.862583409358933) ellipse(.06 and .06/500);
\draw[black, ultra thick, fill=black] (315,0.862706800921043) ellipse(.06 and .06/500);
\draw[black, ultra thick, fill=black] (316,0.862863837694225) ellipse(.06 and .06/500);
\draw[black, ultra thick, fill=black] (317,0.862987188826734) ellipse(.06 and .06/500);
\draw[black, ultra thick, fill=black] (318,0.863140888465439) ellipse(.06 and .06/500);
\draw[black, ultra thick, fill=black] (319,0.863264293894228) ellipse(.06 and .06/500);
\draw[black, ultra thick, fill=black] (320,0.863417258581321) ellipse(.06 and .06/500);
\draw[black, ultra thick, fill=black] (321,0.863538080879384) ellipse(.06 and .06/500);
\draw[black, ultra thick, fill=black] (322,0.86369051933139) ellipse(.06 and .06/500);
\draw[black, ultra thick, fill=black] (323,0.863811227595521) ellipse(.06 and .06/500);
\draw[black, ultra thick, fill=black] (324,0.863960460652438) ellipse(.06 and .06/500);
\draw[black, ultra thick, fill=black] (325,0.864081296068892) ellipse(.06 and .06/500);
\draw[black, ultra thick, fill=black] (326,0.864229854240172) ellipse(.06 and .06/500);
\draw[black, ultra thick, fill=black] (327,0.864348127200065) ellipse(.06 and .06/500);
\draw[black, ultra thick, fill=black] (328,0.864496170412907) ellipse(.06 and .06/500);
\draw[black, ultra thick, fill=black] (329,0.86461442471051) ellipse(.06 and .06/500);
\draw[black, ultra thick, fill=black] (330,0.86475937864726) ellipse(.06 and .06/500);
\draw[black, ultra thick, fill=black] (331,0.86487769314054) ellipse(.06 and .06/500);
\draw[black, ultra thick, fill=black] (332,0.865022019805345) ellipse(.06 and .06/500);
\draw[black, ultra thick, fill=black] (333,0.865137927782628) ellipse(.06 and .06/500);
\draw[black, ultra thick, fill=black] (334,0.865281757493751) ellipse(.06 and .06/500);
\draw[black, ultra thick, fill=black] (335,0.865397588649282) ellipse(.06 and .06/500);
\draw[black, ultra thick, fill=black] (336,0.865538462032924) ellipse(.06 and .06/500);
\draw[black, ultra thick, fill=black] (337,0.865654406020632) ellipse(.06 and .06/500);
\draw[black, ultra thick, fill=black] (338,0.865794680011445) ellipse(.06 and .06/500);
\draw[black, ultra thick, fill=black] (339,0.865908259015452) ellipse(.06 and .06/500);
\draw[black, ultra thick, fill=black] (340,0.866048044126205) ellipse(.06 and .06/500);
\draw[black, ultra thick, fill=black] (341,0.866161614644934) ellipse(.06 and .06/500);
\draw[black, ultra thick, fill=black] (342,0.866298577406302) ellipse(.06 and .06/500);
\draw[black, ultra thick, fill=black] (343,0.866412181242264) ellipse(.06 and .06/500);
\draw[black, ultra thick, fill=black] (344,0.866548569943966) ellipse(.06 and .06/500);
\draw[black, ultra thick, fill=black] (345,0.866659959786038) ellipse(.06 and .06/500);
\draw[black, ultra thick, fill=black] (346,0.8667958781005) ellipse(.06 and .06/500);
\draw[black, ultra thick, fill=black] (347,0.866907211428783) ellipse(.06 and .06/500);
\draw[black, ultra thick, fill=black] (348,0.867040425220795) ellipse(.06 and .06/500);
\draw[black, ultra thick, fill=black] (349,0.867151823323884) ellipse(.06 and .06/500);
\draw[black, ultra thick, fill=black] (350,0.867284487455086) ellipse(.06 and .06/500);
\draw[black, ultra thick, fill=black] (351,0.867393728068351) ellipse(.06 and .06/500);
\draw[black, ultra thick, fill=black] (352,0.867525927831589) ellipse(.06 and .06/500);
\draw[black, ultra thick, fill=black] (353,0.867635149013413) ellipse(.06 and .06/500);
\draw[black, ultra thick, fill=black] (354,0.867764772462733) ellipse(.06 and .06/500);
\draw[black, ultra thick, fill=black] (355,0.867873995692234) ellipse(.06 and .06/500);
\draw[black, ultra thick, fill=black] (356,0.868003088537012) ellipse(.06 and .06/500);
\draw[black, ultra thick, fill=black] (357,0.868110294069741) ellipse(.06 and .06/500);
\draw[black, ultra thick, fill=black] (358,0.868238930778713) ellipse(.06 and .06/500);
\draw[black, ultra thick, fill=black] (359,0.868346070780872) ellipse(.06 and .06/500);
\draw[black, ultra thick, fill=black] (360,0.868472249858604) ellipse(.06 and .06/500);
\draw[black, ultra thick, fill=black] (361,0.868579406784568) ellipse(.06 and .06/500);
\draw[black, ultra thick, fill=black] (362,0.868705079356796) ellipse(.06 and .06/500);
\draw[black, ultra thick, fill=black] (363,0.868810284155368) ellipse(.06 and .06/500);
\draw[black, ultra thick, fill=black] (364,0.868935496195128) ellipse(.06 and .06/500);
\draw[black, ultra thick, fill=black] (365,0.869040657368539) ellipse(.06 and .06/500);
\draw[black, ultra thick, fill=black] (366,0.869163537559609) ellipse(.06 and .06/500);
\draw[black, ultra thick, fill=black] (367,0.869268668144423) ellipse(.06 and .06/500);
\draw[black, ultra thick, fill=black] (368,0.869391051953145) ellipse(.06 and .06/500);
\draw[black, ultra thick, fill=black] (369,0.869494351175332) ellipse(.06 and .06/500);
\draw[black, ultra thick, fill=black] (370,0.869616283851466) ellipse(.06 and .06/500);
\draw[black, ultra thick, fill=black] (371,0.8697194999898) ellipse(.06 and .06/500);
\draw[black, ultra thick, fill=black] (372,0.869839215549209) ellipse(.06 and .06/500);
\draw[black, ultra thick, fill=black] (373,0.869942408742434) ellipse(.06 and .06/500);
\draw[black, ultra thick, fill=black] (374,0.870061643295009) ellipse(.06 and .06/500);
\draw[black, ultra thick, fill=black] (375,0.87016307268932) ellipse(.06 and .06/500);
\draw[black, ultra thick, fill=black] (376,0.870281851745595) ellipse(.06 and .06/500);
\draw[black, ultra thick, fill=black] (377,0.87038321373358) ellipse(.06 and .06/500);
\draw[black, ultra thick, fill=black] (378,0.870499897169705) ellipse(.06 and .06/500);
\draw[black, ultra thick, fill=black] (379,0.870601190213363) ellipse(.06 and .06/500);
\draw[black, ultra thick, fill=black] (380,0.870717398711012) ellipse(.06 and .06/500);
\draw[black, ultra thick, fill=black] (381,0.870817037538749) ellipse(.06 and .06/500);
\draw[black, ultra thick, fill=black] (382,0.870932800305404) ellipse(.06 and .06/500);
\draw[black, ultra thick, fill=black] (383,0.871032342895348) ellipse(.06 and .06/500);
\draw[black, ultra thick, fill=black] (384,0.871146112276467) ellipse(.06 and .06/500);
\draw[black, ultra thick, fill=black] (385,0.871245585297241) ellipse(.06 and .06/500);
\draw[black, ultra thick, fill=black] (386,0.871358891662887) ellipse(.06 and .06/500);
\draw[black, ultra thick, fill=black] (387,0.871456782487462) ellipse(.06 and .06/500);
\draw[black, ultra thick, fill=black] (388,0.871569644403698) ellipse(.06 and .06/500);
\draw[black, ultra thick, fill=black] (389,0.871667438872748) ellipse(.06 and .06/500);
\draw[black, ultra thick, fill=black] (390,0.87177841935016) ellipse(.06 and .06/500);
\draw[black, ultra thick, fill=black] (391,0.871876108122048) ellipse(.06 and .06/500);
\draw[black, ultra thick, fill=black] (392,0.871986630668078) ellipse(.06 and .06/500);
\draw[black, ultra thick, fill=black] (393,0.872082838515389) ellipse(.06 and .06/500);
\draw[black, ultra thick, fill=black] (394,0.872192920200201) ellipse(.06 and .06/500);
\draw[black, ultra thick, fill=black] (395,0.872289007991714) ellipse(.06 and .06/500);
\draw[black, ultra thick, fill=black] (396,0.872397301811067) ellipse(.06 and .06/500);
\draw[black, ultra thick, fill=black] (397,0.872493279942558) ellipse(.06 and .06/500);
\draw[black, ultra thick, fill=black] (398,0.872601133083879) ellipse(.06 and .06/500);
\draw[black, ultra thick, fill=black] (399,0.872695694658062) ellipse(.06 and .06/500);
\draw[black, ultra thick, fill=black] (400,0.872803104985007) ellipse(.06 and .06/500);
\draw[black, ultra thick, fill=black] (401,0.872897541489309) ellipse(.06 and .06/500);
\draw[black, ultra thick, fill=black] (402,0.873003268297731) ellipse(.06 and .06/500);
\draw[black, ultra thick, fill=black] (403,0.87309757291072) ellipse(.06 and .06/500);
\draw[black, ultra thick, fill=black] (404,0.87320285732806) ellipse(.06 and .06/500);
\draw[black, ultra thick, fill=black] (405,0.873295832748329) ellipse(.06 and .06/500);
\draw[black, ultra thick, fill=black] (406,0.873400676979167) ellipse(.06 and .06/500);
\draw[black, ultra thick, fill=black] (407,0.87349350815045) ellipse(.06 and .06/500);
\draw[black, ultra thick, fill=black] (408,0.873596761738986) ellipse(.06 and .06/500);
\draw[black, ultra thick, fill=black] (409,0.873689449462178) ellipse(.06 and .06/500);
\draw[black, ultra thick, fill=black] (410,0.873792272502161) ellipse(.06 and .06/500);
\draw[black, ultra thick, fill=black] (411,0.873883692456329) ellipse(.06 and .06/500);
\draw[black, ultra thick, fill=black] (412,0.87398607534751) ellipse(.06 and .06/500);
\draw[black, ultra thick, fill=black] (413,0.874077350431866) ellipse(.06 and .06/500);
\draw[black, ultra thick, fill=black] (414,0.874178231235825) ellipse(.06 and .06/500);
\draw[black, ultra thick, fill=black] (415,0.87426934301907) ellipse(.06 and .06/500);
\draw[black, ultra thick, fill=black] (416,0.874369791308834) ellipse(.06 and .06/500);
\draw[black, ultra thick, fill=black] (417,0.874459711045196) ellipse(.06 and .06/500);
\draw[black, ultra thick, fill=black] (418,0.874559729907602) ellipse(.06 and .06/500);
\draw[black, ultra thick, fill=black] (419,0.874649485498653) ellipse(.06 and .06/500);
\draw[black, ultra thick, fill=black] (420,0.874748083128084) ellipse(.06 and .06/500);
\draw[black, ultra thick, fill=black] (421,0.874837661622929) ellipse(.06 and .06/500);
\draw[black, ultra thick, fill=black] (422,0.874935837202) ellipse(.06 and .06/500);
\draw[black, ultra thick, fill=black] (423,0.875024290104258) ellipse(.06 and .06/500);
\draw[black, ultra thick, fill=black] (424,0.875122031057477) ellipse(.06 and .06/500);
\draw[black, ultra thick, fill=black] (425,0.875210313404668) ellipse(.06 and .06/500);
\draw[black, ultra thick, fill=black] (426,0.875306712738004) ellipse(.06 and .06/500);
\draw[black, ultra thick, fill=black] (427,0.875394803032488) ellipse(.06 and .06/500);
\draw[black, ultra thick, fill=black] (428,0.875490786029875) ellipse(.06 and .06/500);
\draw[black, ultra thick, fill=black] (429,0.87557781142218) ellipse(.06 and .06/500);
\draw[black, ultra thick, fill=black] (430,0.87567336746903) ellipse(.06 and .06/500);
\draw[black, ultra thick, fill=black] (431,0.875760205889876) ellipse(.06 and .06/500);
\draw[black, ultra thick, fill=black] (432,0.875854492928707) ellipse(.06 and .06/500);
\draw[black, ultra thick, fill=black] (433,0.875941133527598) ellipse(.06 and .06/500);
\draw[black, ultra thick, fill=black] (434,0.876035008650172) ellipse(.06 and .06/500);
\draw[black, ultra thick, fill=black] (435,0.876120642541942) ellipse(.06 and .06/500);
\draw[black, ultra thick, fill=black] (436,0.876214086463075) ellipse(.06 and .06/500);
\draw[black, ultra thick, fill=black] (437,0.876299525944739) ellipse(.06 and .06/500);
\draw[black, ultra thick, fill=black] (438,0.876391778805833) ellipse(.06 and .06/500);
\draw[black, ultra thick, fill=black] (439,0.876477005209529) ellipse(.06 and .06/500);
\draw[black, ultra thick, fill=black] (440,0.876568849447096) ellipse(.06 and .06/500);
\draw[black, ultra thick, fill=black] (441,0.87665312050723) ellipse(.06 and .06/500);
\draw[black, ultra thick, fill=black] (442,0.876744540111389) ellipse(.06 and .06/500);
\draw[black, ultra thick, fill=black] (443,0.876828611838286) ellipse(.06 and .06/500);
\draw[black, ultra thick, fill=black] (444,0.876918901577402) ellipse(.06 and .06/500);
\draw[black, ultra thick, fill=black] (445,0.877002750905863) ellipse(.06 and .06/500);
\draw[black, ultra thick, fill=black] (446,0.877092635259447) ellipse(.06 and .06/500);
\draw[black, ultra thick, fill=black] (447,0.877175582595219) ellipse(.06 and .06/500);
\draw[black, ultra thick, fill=black] (448,0.877265046927331) ellipse(.06 and .06/500);
\draw[black, ultra thick, fill=black] (449,0.877347782608739) ellipse(.06 and .06/500);
\draw[black, ultra thick, fill=black] (450,0.877436183951001) ellipse(.06 and .06/500);
\draw[black, ultra thick, fill=black] (451,0.877518685512955) ellipse(.06 and .06/500);
\draw[black, ultra thick, fill=black] (452,0.877606684097839) ellipse(.06 and .06/500);
\draw[black, ultra thick, fill=black] (453,0.87768833506274) ellipse(.06 and .06/500);
\draw[black, ultra thick, fill=black] (454,0.87777591681251) ellipse(.06 and .06/500);
\draw[black, ultra thick, fill=black] (455,0.877857350890862) ellipse(.06 and .06/500);
\draw[black, ultra thick, fill=black] (456,0.877943923085006) ellipse(.06 and .06/500);
\draw[black, ultra thick, fill=black] (457,0.878025113814512) ellipse(.06 and .06/500);
\draw[black, ultra thick, fill=black] (458,0.878111293947889) ellipse(.06 and .06/500);
\draw[black, ultra thick, fill=black] (459,0.878191679032726) ellipse(.06 and .06/500);
\draw[black, ultra thick, fill=black] (460,0.87827744507874) ellipse(.06 and .06/500);
\draw[black, ultra thick, fill=black] (461,0.878357601009235) ellipse(.06 and .06/500);
\draw[black, ultra thick, fill=black] (462,0.878442415997329) ellipse(.06 and .06/500);
\draw[black, ultra thick, fill=black] (463,0.878522326917683) ellipse(.06 and .06/500);
\draw[black, ultra thick, fill=black] (464,0.878606749851369) ellipse(.06 and .06/500);
\draw[black, ultra thick, fill=black] (465,0.878685897247036) ellipse(.06 and .06/500);
\draw[black, ultra thick, fill=black] (466,0.878769908719976) ellipse(.06 and .06/500);
\draw[black, ultra thick, fill=black] (467,0.87884882363393) ellipse(.06 and .06/500);
\draw[black, ultra thick, fill=black] (468,0.878931938683702) ellipse(.06 and .06/500);
\draw[black, ultra thick, fill=black] (469,0.8790105969677) ellipse(.06 and .06/500);
\draw[black, ultra thick, fill=black] (470,0.879093327325287) ellipse(.06 and .06/500);
\draw[black, ultra thick, fill=black] (471,0.879171262250723) ellipse(.06 and .06/500);
\draw[black, ultra thick, fill=black] (472,0.879253584538263) ellipse(.06 and .06/500);
\draw[black, ultra thick, fill=black] (473,0.879331281883762) ellipse(.06 and .06/500);
\draw[black, ultra thick, fill=black] (474,0.879412756412285) ellipse(.06 and .06/500);
\draw[black, ultra thick, fill=black] (475,0.879490195782171) ellipse(.06 and .06/500);
\draw[black, ultra thick, fill=black] (476,0.879571285910876) ellipse(.06 and .06/500);
\draw[black, ultra thick, fill=black] (477,0.879648037442789) ellipse(.06 and .06/500);
\draw[black, ultra thick, fill=black] (478,0.879728728740615) ellipse(.06 and .06/500);
\draw[black, ultra thick, fill=black] (479,0.879805238809953) ellipse(.06 and .06/500);
\draw[black, ultra thick, fill=black] (480,0.879885128729821) ellipse(.06 and .06/500);
\draw[black, ultra thick, fill=black] (481,0.879961369878469) ellipse(.06 and .06/500);
\draw[black, ultra thick, fill=black] (482,0.880040882104111) ellipse(.06 and .06/500);
\draw[black, ultra thick, fill=black] (483,0.88011647764701) ellipse(.06 and .06/500);
\draw[black, ultra thick, fill=black] (484,0.880195592691343) ellipse(.06 and .06/500);
\draw[black, ultra thick, fill=black] (485,0.880270940195997) ellipse(.06 and .06/500);
\draw[black, ultra thick, fill=black] (486,0.880349296152301) ellipse(.06 and .06/500);
\draw[black, ultra thick, fill=black] (487,0.880424374616328) ellipse(.06 and .06/500);
\draw[black, ultra thick, fill=black] (488,0.880502359689373) ellipse(.06 and .06/500);
\draw[black, ultra thick, fill=black] (489,0.880576820874563) ellipse(.06 and .06/500);
\draw[black, ultra thick, fill=black] (490,0.880654415924603) ellipse(.06 and .06/500);
\draw[black, ultra thick, fill=black] (491,0.880728626411034) ellipse(.06 and .06/500);
\draw[black, ultra thick, fill=black] (492,0.880805503030194) ellipse(.06 and .06/500);
\draw[black, ultra thick, fill=black] (493,0.880879440506974) ellipse(.06 and .06/500);
\draw[black, ultra thick, fill=black] (494,0.880955950417238) ellipse(.06 and .06/500);
\draw[black, ultra thick, fill=black] (495,0.881029306833476) ellipse(.06 and .06/500);
\draw[black, ultra thick, fill=black] (496,0.881105428825383) ellipse(.06 and .06/500);
\draw[black, ultra thick, fill=black] (497,0.881178528160929) ellipse(.06 and .06/500);
\draw[black, ultra thick, fill=black] (498,0.881253973885509) ellipse(.06 and .06/500);
\draw[black, ultra thick, fill=black] (499,0.881326799748092) ellipse(.06 and .06/500);
\draw[black, ultra thick, fill=black] (500,0.881401884269429) ellipse(.06 and .06/500);
\draw[black, ultra thick, fill=black] (501,0.881474152544609) ellipse(.06 and .06/500);
\draw[black, ultra thick, fill=black] (502,0.881548855652134) ellipse(.06 and .06/500);
\draw[black, ultra thick, fill=black] (503,0.881620870873875) ellipse(.06 and .06/500);
\draw[black, ultra thick, fill=black] (504,0.88169493180628) ellipse(.06 and .06/500);
\draw[black, ultra thick, fill=black] (505,0.88176666869163) ellipse(.06 and .06/500);
\draw[black, ultra thick, fill=black] (506,0.881840372107423) ellipse(.06 and .06/500);
\draw[black, ultra thick, fill=black] (507,0.881911583413063) ellipse(.06 and .06/500);
\draw[black, ultra thick, fill=black] (508,0.881984912978703) ellipse(.06 and .06/500);
\end{tikzpicture}
\caption{Plot of the values of 
$\frac{\ln f(n)}{(6n)^{1/3}\ln[(6n)^{1/3}]}$ for $31 \leq n \leq 508$.
}\label{valuesplot}
\end{figure}
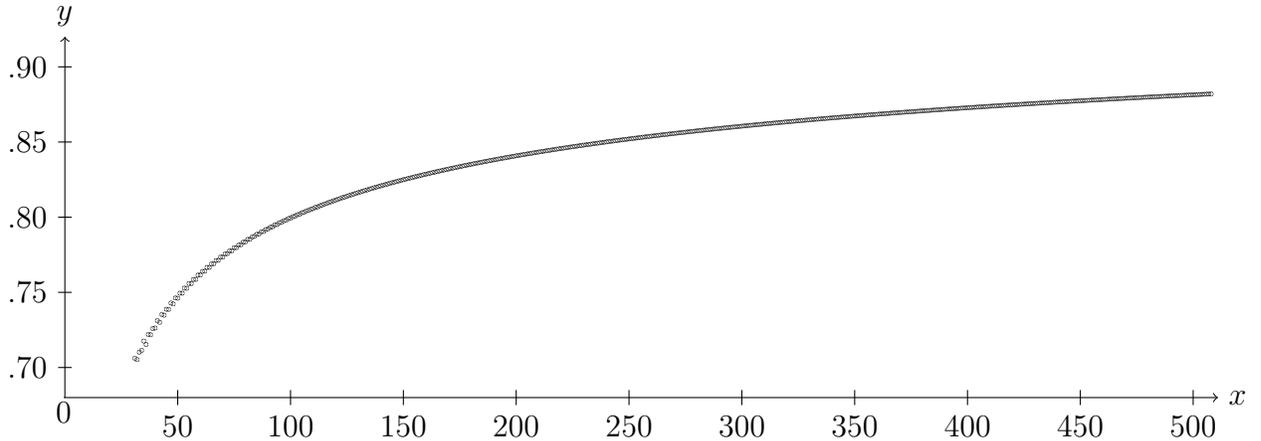
\end{center}

\section{Proofs}
\begin{lemma}\label{upperboundRlemma}
The number $r$ of distinct multiplicities in a Wilf partition of $n$ is at most
$(6 n)^{1/3}$.
\end{lemma}
\begin{proof}
For a given positive integer~$r$,
the smallest possible~$n$ admitting a Wilf partition with~$r$ distinct
multiplicities is obtained by taking:
\begin{itemize}
\item multiplicity $m_{1} = r$ for part $p_{1} = 1$;
\item multiplicity $m_{2} = r-1$ for part $p_{2} = 2$;
\item multiplicity $m_{3} = r-2$ for part $p_{3} = 3$;
\item \qquad \vdots
\item multiplicity $m_{r} = 1$ for part $p_{r} = r$.
\end{itemize}
This yields
$$n = \sum_{i = 1}^r (r+ 1 - i) i 
= \frac{1}{6}r^{3} + \frac{1}{2}r^{2} + \frac{1}{3}r.$$
Hence $r \leq (6 n)^{1/3}$.
This completes the proof of Lemma~\ref{upperboundRlemma}.
\end{proof}

\begin{lemma}\label{upperbound}
An upper bound for $\ln{f(n)}$ is
$$
\ln{f(n)} \le (1 + o(1))  \frac{6^{1/3}}{3} n^{1/3}\ln{n} .
$$
\end{lemma}
\begin{proof}
For each Wilf partition of~$n$,
put the terms in decreasing order according to the values of the
products $m_{i}p_{i}$.  If two terms are equal,
break the tie by writing in decreasing order of the multiplicities.
This gives a canonical way to write the Wilf partitions.  For instance,
$$27 = 8 + 3 + 3 + 2 + 2 + 2 + 1 + 1 + 1 + 1 + 1 + 1 + 1$$
can be written as
$$27 = (1 \times 8) + (7 \times 1) + (3 \times 2) + (2 \times 3) = m_1 p_1 + m_2 p_2 + m_3
p_3 + m_4 p_4.$$
Notice that the products are, respectively 8, 7, 6, 6.  The
ordering of the two terms each with product~$6$ 
was decided by writing those terms in decreasing order of the multiplicities.

With this representation in mind, it follows that 
the number $f(n,r)$
of Wilf partitions of~$n$ with $r$ distinct multiplicities is no larger than
$$
p(n, r) \times [\max \{d(j): 1 \leq j \leq n\}]^r.
$$

Severin~\cite{Wigert1907}
[see also \cite[Th.~317, Chap.~XVIII.1]{HardyWright}] 
showed that
$$
\limsup_{n\rightarrow\infty}\frac{\ln d(n)}{(\ln{n}) / (\ln{\ln{n}})}
= \ln{2}.
$$
Therefore, there exists a constant $C$ such that, provided~$n$ is sufficiently large,
$f(n,r)$
is no larger than
\begin{equation}
  \label{fnr}
p(n, r) \times \max_{3\leq j\leq n} \left\{\exp\left(r C \frac{\ln j}{\ln \ln j}\right)\right\}
= p(n, r) \times \exp\left(r C \frac{\ln n}{\ln \ln n}\right).
\end{equation}

Since by Lemma~\ref{upperboundRlemma} we have $r \leq (6n)^{1/3}$ if 
$f(n, r) > 0$,
the second factor here does not contribute to the lead-order logarithmic asymptotics for $f(n)$.  
Now we utilize
Exercise 7.2.1.4-34 in \citet{Knuth4A},
which concerns $p(n, r)$ and is stated (in our notation)
for $r \leq n^{1/3}$; but  Knuth's argument is easily checked to 
hold also for $r \leq (c n)^{1/3}$ for any constant~$c$.
Choosing $c = 6$ we find, for all $r \leq (6 n)^{1/3}$,
$$
p(n, r) = O\left(\frac{n^{r - 1}}{r! (r - 1)!}\right) \leq \exp\left[(1 + o(1)) (1/3) (6 n)^{1/3} \ln n\right]
$$
as $n \to \infty$,
and then \eqref{fnr} yields the same estimate for $f(n,r)$.  
The proof of Lemma~\ref{upperbound} is completed by a summation over
$r\le(6n)^{1/3}$. 
\end{proof}

\begin{lemma}\label{lowerbound}
A lower bound for $\ln{f(n)}$ is
$$\ln{f(n)} \ge (1 + o(1)) \frac{6^{1/3}}{3} n^{1/3}\ln{n} .$$
\end{lemma}
\begin{proof}
Let $a <  6^{1/3}$, and let $K$ be a fixed large integer.
Let $b:=\lfloor an^{1/3}/K\rfloor$ and
divide the interval $[1,Kb]\subseteq[1, a n^{1/3}]$ into $K$ equal parts
$I_1,\ldots,I_K$. 
Consider only permutations 
$(p_1,\dots,p_{Kb})$ 
of $[1, Kb]$ that map $I_j$ into $I_{K + 1 - j}$ for every $j \in \{1, \dots, K\}$.
For such permutations, if 
$$
a = [6 (1 - 2 \epsilon)]^{1/3},
$$
and if~$K$ is large enough (depending on $\epsilon$ but not on
$n$), then  $\sum_i i p_i < (1-\epsilon)n$, and we obtain
(if~$n$ is large enough) 
a Wilf partition
by taking $i$ parts of size $p_i$ for each $2\le i\le Kb$ and a single part of size
$n - \sum_{i=2}^{Kb} i p_i$.

Thus the number of Wilf partitions of $n$ is at least
$$
b!^K
=\exp\left[ Kb (\ln b + O(1)) \right]
=\exp\left[ a n^{1/3} (\ln n^{1/3} + O(1)) \right]
=\exp\left[ \left( \frac{a}{3} + o(1) \right) n^{1/3} \ln n \right].
$$ 
This completes the proof of Lemma~\ref{lowerbound}, since we may take $a$
arbitrarily close to $6^{1/3}$.
  \end{proof}

\section{Open Problems}
We have found the first-order asymptotic description of
$\ln{f(n)}$, but lower-order terms of $\ln{f(n)}$ remain unknown.
A Herculean task would be to find the first-order asymptotic
description of $f(n)$ itself.
As a much simpler task, 
David~S. Newman has mentioned that it would be nice to have a 
proof that $f(n)$ is nondecreasing.

\subsection{An involution}

\cite{WilfPDF}
mentions a mapping on $T(n)$ in which the roles of
parts and multiplicities are interchanged.  We let $\sigma_{n}$
denote this mapping.  
Hence,
$\sigma_n(({\bf m},{\bf p}))$ is 
$$
({\bf p}_\pi,{\bf m}_\pi)
=((p_{\pi(1)},\dots,p_{\pi(r)}),(m_{\pi(1)},\dots,m_{\pi(r)})),
$$
where~$\pi$ is the permutation making 
$m_{\pi(1)}<\dots<m_{\pi(r)}$.
For instance, two partitions of 83 are
$$
83  = 1 + 1 + 1 + 1 + 1 + 1 + 1 + 4 + 4 + 4 + 4 + 5 + 5 + 5 + 5 + 5 + 5 + 5 +
5 + 5 + 5 + 5 + 5
$$
and
$$
83 = 4 + 4 + 4 + 4 + 7 + 12 + 12 + 12 + 12 + 12.
$$
More succinctly, we can write these partitions as elements of
$T(83)$ as $((7,4,12),(1,4,5))$ and $((4,1,5),(4,7,12))$.
Then we have
$$\sigma_{83} : ((7,4,12),(1,4,5)) \mapsto ((4,1,5),(4,7,12)),$$
and, reversing the roles of the multiplicities and partitions, we have
$$\sigma_{83} : ((4,1,5),(4,7,12)) \mapsto ((7,4,12),(1,4,5)).$$
Switching the roles of the parts and multiplicities again,
we get back to the original partition, so $\sigma_{n}$ is seen to be an
involution.  Note that $\sigma_{n}$ has order~$2$ on most elements of
$T(n)$, but $\sigma_{n}$ has order~$1$ on some elements.
In particular, $\sigma_n$ fixes every $({\bf m},{\bf p})$
for which $m_{i} = p_{i}$ for all $i$, for example, 
$$\sigma_{65} : ((2,5,6),(2,5,6)) \mapsto ((2,5,6),(2,5,6));$$
but there are also other examples, such as 
$$\sigma_{10} : ((1,2,3), (3,2,1)) \mapsto ((1,2,3), (3,2,1)).
$$
It is an open problem to find (asymptotics for) the number of fixed points
of $\sigma_n$.

\section*{Acknowledgements}
Daniel Kane has independently obtained similar results.
We thank David~S. Newman and Vladeta Jovovic
for early conversations about efficient ways to generate
the numbers in Wilf's sequence.
We also thank David~S. Newman for early feedback on the manuscript.
The sequence can be accessed in Neil Sloane's database,
The Online Encyclopedia of Integer Sequences
[\cite{OEIS}]. 
Maciej~I.\ Wilczy{\'n}ski has recently generated the first 500 
values in Wilf's sequence.
The current list of known values of Wilf's sequence is
found at \url{http://oeis.org/A098859/b098859.txt}.


\end{document}